\documentclass[11pt]{amsart}
\usepackage{amsopn}
\usepackage{amssymb, amscd}
\usepackage{paracol}
\usepackage{graphicx, graphics, epsfig}
\usepackage{faktor}
\usepackage{caption}
\usepackage{enumitem}
\usepackage{hyperref}
\usepackage{tikz-cd}
\globalcounter{figure}
\usepackage{xcolor}

\graphicspath{ {imagenes/} }

\newcommand{\nc}{\newcommand}

\nc{\fg}{\mathfrak{f} } \nc{\vg}{\mathfrak{v} } \nc{\wg}{\mathfrak{w} }
\nc{\zg}{\mathfrak{z} } \nc{\ngo}{\mathfrak{n} } \nc{\kg}{\mathfrak{k} }
\nc{\mg}{\mathfrak{m} } \nc{\bg}{\mathfrak{b} } \nc{\ggo}{\mathfrak{g} }
\nc{\ggob}{\overline{\mathfrak{g}} } \nc{\sog}{\mathfrak{so} }
\nc{\sug}{\mathfrak{su} } \nc{\spg}{\mathfrak{sp} } \nc{\slg}{\mathfrak{sl} }
\nc{\glg}{\mathfrak{gl} } \nc{\cg}{\mathfrak{c} } \nc{\rg}{\mathfrak{r} }
\nc{\hg}{\mathfrak{h} } \nc{\tgo}{\mathfrak{t} } \nc{\ug}{\mathfrak{u} }
\nc{\dg}{\mathfrak{d} } \nc{\ag}{\mathfrak{a} } \nc{\pg}{\mathfrak{p} }
\nc{\sg}{\mathfrak{s} } \nc{\affg}{\mathfrak{aff} } \nc{\qg}{\mathfrak{q} }
\nc{\lgo}{\mathfrak{l} }

\nc{\pca}{\mathcal{P}} \nc{\nca}{\mathcal{N}} \nc{\lca}{\mathcal{L}}
\nc{\oca}{\mathcal{O}} \nc{\mca}{\mathcal{M}} \nc{\tca}{\mathcal{T}}
\nc{\aca}{\mathcal{A}} \nc{\cca}{\mathcal{C}} \nc{\gca}{\mathcal{G}}
\nc{\sca}{\mathcal{S}} \nc{\hca}{\mathcal{H}} \nc{\bca}{\mathcal{B}}
\nc{\dca}{\mathcal{D}} \nc{\val}{\operatorname{val}}

\nc{\vp}{\varphi} \nc{\ddt}{\frac{d}{dt}} \nc{\dds}{\frac{d}{ds}}
\nc{\dpar}{\frac{\partial}{\partial t}} \nc{\im}{\mathrm{i}}

\nc{\SO}{\mathrm{SO}} \nc{\Spe}{\mathrm{Sp}} \nc{\Sl}{\mathrm{SL}}
\nc{\SU}{\mathrm{SU}} \nc{\Or}{\mathrm{O}} \nc{\U}{\mathrm{U}} \nc{\Gl}{\mathrm{GL}}
\nc{\Se}{\mathrm{S}} \nc{\Cl}{\mathrm{Cl}} \nc{\Spein}{\mathrm{Spin}}
\nc{\Pin}{\mathrm{Pin}} \nc{\G}{\mathrm{GL}_n(\RR)} \nc{\g}{\mathfrak{gl}_n(\RR)}

\nc{\RR}{{\Bbb R}} \nc{\HH}{{\Bbb H}} \nc{\CC}{{\Bbb C}} \nc{\ZZ}{{\Bbb Z}}
\nc{\FF}{{\Bbb F}} \nc{\NN}{{\Bbb N}} \nc{\QQ}{{\Bbb Q}} \nc{\PP}{{\Bbb P}} \nc{\OO}{{\Bbb O}}

\nc{\vs}{\vspace{.2cm}} \nc{\vsp}{\vspace{1cm}} \nc{\ip}{\langle\cdot,\cdot\rangle}
\nc{\ipp}{(\cdot,\cdot)} \nc{\la}{\langle} \nc{\ra}{\rangle} \nc{\unm}{\tfrac{1}{2}}
\nc{\unc}{\tfrac{1}{4}} \nc{\und}{\tfrac{1}{16}} \nc{\no}{\vs\noindent}
\nc{\lam}{\Lambda^2(\RR^n)^*\otimes\RR^n} \nc{\tangz}{{\rm T}^{\rm Zar}}
\nc{\nor}{{\sf n}}  \nc{\mum}{/\!\!/} \nc{\kir}{/\!\!/\!\!/}
\nc{\Ri}{\tfrac{4\Ric_{\mu}}{||\mu||^2}} \nc{\ds}{\displaystyle}
\nc{\ben}{\begin{enumerate}} \nc{\een}{\end{enumerate}} \nc{\f}{\frac}
\nc{\lb}{[\cdot,\cdot]} \nc{\isn}{\tfrac{1}{||v||^2}}
\nc{\gkp}{(\ggo=\kg\oplus\pg,\ip)} \nc{\ukh}{(\ug=\kg\oplus\hg,\ip)}
\nc{\tgkp}{(\tilde{\ggo}=\kg\oplus\pg,\ip)}
\nc{\wt}{\widetilde}
\nc{\iop}{\mathtt{i}} \nc{\jop}{\mathtt{j}} \nc{\gk}{g_{\kil}}

\nc{\alp}{\alpha^2}  \nc{\bet}{\beta^2}  \nc{\gam}{\gamma^2} \nc{\et}{\eta^2}

\nc{\Hess}{\operatorname{Hess}} \nc{\ad}{\operatorname{ad}}
\nc{\Ad}{\operatorname{Ad}} \nc{\rank}{\operatorname{rank}}
\nc{\Irr}{\operatorname{Irr}} \nc{\End}{\operatorname{End}}
\nc{\Aut}{\operatorname{Aut}} \nc{\Inn}{\operatorname{Inn}}
\nc{\Der}{\operatorname{Der}} \nc{\Ker}{\operatorname{Ker}}
\nc{\Iso}{\operatorname{Iso}} \nc{\Diff}{\operatorname{Diff}}
\nc{\Lie}{\operatorname{L}} \nc{\tr}{\operatorname{tr}} \nc{\dif}{\operatorname{d}}
\nc{\sen}{\operatorname{sen}} \nc{\modu}{\operatorname{mod}}
\nc{\CRic}{\operatorname{PP}} \nc{\Cric}{\operatorname{P}} \nc{\Ricci}{\operatorname{Ric}}
\nc{\sym}{\operatorname{sym}} \nc{\herm}{\operatorname{herm}} \nc{\symac}{\operatorname{sym^{ac}}}
\nc{\symc}{\operatorname{sym^{c}}} \nc{\scalar}{\operatorname{Sc}}
\nc{\grad}{\operatorname{grad}} \nc{\ricci}{\operatorname{Rc}}
\nc{\Nor}{\operatorname{Norm}}  \nc{\ricc}{\operatorname{Rc^{c}}}
\nc{\Ricc}{\operatorname{Ric^{c}}} \nc{\ricac}{\operatorname{Rc^{ac}}}
\nc{\Ricac}{\operatorname{Ric^{ac}}} \nc{\Riem}{\operatorname{Rm}} \nc{\Sec}{\operatorname{Sec}}
\nc{\riccig}{\operatorname{ric^{\gamma}}} \nc{\Rin}{\operatorname{M}}
\nc{\kil}{\operatorname{B}} \nc{\cas}{\operatorname{C}} \nc{\Le}{\operatorname{L}}
\nc{\tang}{\operatorname{T}}
\nc{\level}{\operatorname{level}} \nc{\rad}{\operatorname{r}}
\nc{\abel}{\operatorname{ab}} \nc{\CH}{\operatorname{CH}} \nc{\Cone}{{\mathcal C}} \nc{\CCone}{\operatorname{CC}} \nc{\CP}{{\mathcal P}}
\nc{\mcc}{\operatorname{mcc}} \nc{\Adj}{\operatorname{Adj}}
\nc{\Order}{\operatorname{O}}  \nc{\inj}{\operatorname{inj}} \nc{\proy}{\operatorname{pr}}
\nc{\vol}{\operatorname{vol}} \nc{\Diag}{\operatorname{Dg}} \nc{\Diagg}{\operatorname{Diag}}
\nc{\Spec}{\operatorname{Spec}} \nc{\Ima}{\operatorname{Im}} \nc{\Rea}{\operatorname{Re}}
\nc{\spann}{\operatorname{span}} \nc{\Aff}{\operatorname{Aff}}
\nc{\mm}{\operatorname{m}} \nc{\id}{\operatorname{Id}} \nc{\Bb}{\operatorname{B}}
\nc{\lic}{\operatorname{L}}

\newtheorem{theorem}{Theorem}[section]
\newtheorem{proposition}[theorem]{Proposition}

\newtheorem{lemma}[theorem]{Lemma}

\theoremstyle{definition}
\newtheorem{definition}[theorem]{Definition}

\theoremstyle{remark}
\newtheorem{remark}[theorem]{Remark}

\newtheorem{assumption}[theorem]{Assumption}

\title{Generalized Ricci Flow on aligned homogeneous spaces}

\author{Valeria Guti\'errez}

\address{FAMAF, Universidad Nacional de C\'ordoba and CIEM, CONICET (Argentina)}
\email{valeria.gutierrez@unc.edu.ar}

\thanks{
This research was partially supported by a CONICET doctoral fellowship.}

\begin{document}

\begin{abstract}
The fixed points of the generalized Ricci flow are the Bismut Ricci flat metrics, i.e., a generalized metric $(g,H)$ on a manifold $M$, where $g$ is a Riemannian metric and $H$ a closed $3$-form, such that $H$ is $g$-harmonic and $\ricci(g)=\unc H_g^2$. Given two standard Einstein homogeneous spaces $G_i/K$, where each $G_i$ is a compact simple Lie group and $K$ is a closed subgroup of them holding some extra assumption, we consider $M=G_1\times G_2/\Delta K$. Recently, Lauret and Will proved the existence of a Bismut Ricci flat metric on any of these spaces. We proved that this metric is always asymptotically stable for the generalized Ricci flow  on $M$ among a subset of $G$-invariant metrics and, if $G_1=G_2$, then it is globally stable.
\end{abstract}

 \maketitle
 \tableofcontents

\section{Introduction}
In the context of generalized Riemannian geometry has arisen an extension of the Ricci flow equation called {\it generalized Ricci flow}. Given a manifold $M$ and a generalized metric encoded in the pair $(g,H)$, where $g$ is a Riemannian metric and $H$ a closed $3$-form on $M$, this flow studied in \cite{GrcStr} is given by
\begin{equation}\label{GRF}
\left\{\begin{array}{l}
\dpar g(t) = -2\ricci(g(t)) +\unm (H(t))_{g(t)}^2, \\ \\
\dpar H(t) = -dd_{g(t)}^*H(t),
\end{array}\right.
\end{equation}
where $H_g^2:=g(\iota_\cdot H,\iota_\cdot H)$ and $d_g^*$ is the adjoint of $d$ with respect to $g$.

A pair $(g,H)$ is naturally associated with {\it Bismut} connections, i.e., the unique metric connection on the Riemannian manifold $(M,g)$ with torsion equal to $H$, in this sense the generalized Ricci flow is the natural evolution in the direction of the Ricci tensor of this connection.

In search of canonical generalized geometry, Garc\'ia-Fernandez and Streets in \cite{GrcStr} proposed to look for the fixed points of this flow, the so-called {\it Bismut Ricci flat} (BRF for short) generalized metrics or generalized Einstein metrics, i.e., $(g,H)$ such that
\begin{equation*}
\ricci(g)=\unc H_g^2 \quad\mbox{and}\quad \mbox{$H$ is $g$-harmonic}.
\end{equation*}

Let $G_1, G_2$ be compact, connected, simple Lie groups and $K\subset G_1,G_2$ a closed Lie subgroup. Suppose that for $i=1,2$, there exist constants $0<a_1\leq a_2<1$ such that $\kil_{\kg}=a_i\kil_{\ggo_i}|_{\kg}$, where $\kg,\ggo_i$ are the Lie algebras of $K$ and $G_i$, respectively,  and $\kil_{\hg}$ is the Killing form of the Lie algebra $\hg$.  The homogeneous space defined by $M = G_1 \times G_2/\Delta K$ is aligned with $(c_1,c_2)=(\tfrac{a_1+a_2}{a_2},\tfrac{a_1+a_2}{a_1})$ and $\lambda_1=\dots=\lambda_t=\tfrac{a_1a_2}{a_1+a_2}$, as defined in \cite{H3}, and so its third Betti number is one.  Recently in \cite{BRFs}, Lauret and Will found a BRF $G$-invariant generalized metric on any aligned homogeneous space $M = G/K$, where $G$ is a compact semisimple Lie group with two simple factors generalizing results obtained in \cite{PdsRff1, PdsRff2}.

If $G:=G_1 \times G_2$, then we consider its Killing metric given by
\begin{equation}\label{gkil}
g_{\kil}=(-\Bb_{\ggo_1})+(-\Bb_{\ggo_2}),
\end{equation}
and let $\ggo=\kg\oplus\pg$ be the $g_{\kil}$-orthogonal reductive decomposition of $\ggo$. We fix the standard metric of $M = G_1 \times G_2/\Delta K$ determined by $g_{\kil}|_{\pg\times \pg}$ as a background metric and consider the $g_{\kil}$-orthogonal Ad($K$)-invariant decomposition
$$\pg=\pg_1\oplus\pg_2\oplus\pg_3,$$
where each $\pg_i$ is equivalent to the isotropy representation of the homogeneous space $M_i=G_i/K$ for $i=1,2$ and $\pg_3$ is equivalent to the adjoint representation $\kg$ (see \cite[Proposition 5.1]{H3}).

Following \cite{H3}, the closed $3$-form $H_0$ on M given by
{\footnotesize$$H_0(X,Y,Z) := Q([X,Y],Z) + Q([X,Y]_\kg,Z) - Q([X,Z]_\kg,Y) + Q([Y,Z]_\kg,X)  \quad  \text{for all\ }  X,Y,Z \in \pg,$$}
where $Q=\Bb_{\ggo_1}-\tfrac{a_2}{a_1}\Bb_{\ggo_2}$, is $g$-harmonic for any $G$-invariant metric $g=(x_1,x_2,x_3)_{g_{\kil}}$ of the form:
\begin{equation} \label{gg}
 g:=x_1g_{\kil}|_{\pg_1 \times \pg_1} + x_2g_{\kil}|_{\pg_2 \times \pg_2}+x_3g_{\kil}|_{\pg_3 \times \pg_3}, \qquad x_1,x_2,x_3>0,
\end{equation}
these metrics wil be called \textit{diagonal}.

The $G$-invariant BRF generalized metric found in \cite{BRFs} is given, up to scaling, by $(g_0,H_0)$ (see \cite[Remark A.3]{Corr}), where
\begin{equation}\label{brf-met}
g_0:=(1,\tfrac{a_2}{a_1},\tfrac{a_1+a_2}{a_1})_{g_{\kil}}.
\end{equation}

In this paper, we study the dynamical stability of this generalized metric as a fixed point of the generalized Ricci flow given in \eqref{GRF} on homogeneous spaces of the form $M = G_1 \times G_2/\Delta K$ as above, such that the standard metric on $M_i=G_i/K$ is Einstein  for $i=1,2$. Note that since each $G_i$ is simple, the spaces $M_i$ are given in \cite{Bss} (see also \cite{stab} and \cite{Rical}), there are $17$ families and $50$ isolated examples among irreducible symmetric, isotropy irreducible and non-isotropy irreducible homogeneous spaces. Note that we can consider $G_1=G_2=H$, in this case $a_1=a_2$. Our main result is the following.
\begin{theorem}
  Let $M = G_1 \times G_2/\Delta K$  be a homogeneous space as above such that $\kil_{\kg}=a_i \kil_{\ggo_i}|_{\kg}$ and $(M_i=G_i/K, g_{\kil}^i)$ is Einstein, where $g_{\kil}^i$ is the standard metric on each $M_i$ for $i=1,2$, then
  \begin{enumerate}[{label=(\roman*)}]
\item There exists a neighborhood $U$ of the metric $g_0=(1,\tfrac{a_2}{a_1},\tfrac{a_1+a_2}{a_1})_{g_{\kil}}$ in the space of all diagonal metrics, such that the generalized Ricci flow converges to $(g_0,H_0)$ starting at any generalized metric $(g,H_0)$ with $g$ in $U$. (See Theorem \ref{stable z11} and Proposition \ref{floww}).
\item   Let $M=H\times H/\Delta K$ (i.e., $a_1=a_2$) be a homogeneous space such that $(H/K, g_{\kil})$ is  Einstein and  $\kil_{\kg}=a\kil_{\hg}|_{\kg}$, then any diagonal generalized Ricci flow solution converges to the Bismut Ricci flat metric $(g_0,H_0)$, where $g_0=(1,1,2)_{g_{\kil}}$. (See Theorem  \ref{global}).
\end{enumerate}
\end{theorem}

\begin{remark}
If $M = G_1 \times G_2/\Delta K$ is multiplicity-free, i.e., $M_1,M_2$ are both isotropy irreducible, $K$ is simple, the Ad($K$)-representations $\pg_1$, $\pg_2$ are inequivalent and neither of them is equivalent to the adjoint representation $\kg$, then the metrics of the form given in \eqref{gg} are all the $G$-invariant metrics on the homogeneous space $M$.
\end{remark}

Finally, in Section \ref{so}, we give an overview on the generalized Ricci flow and its fixed points on simple compact Lie groups and an analysis of the stability on $\SO(n)$, the only known simple Lie group admitting a nice basis. Our result is the following.
\begin{theorem}
There exists a neighborhood $U$ of the Killing metric $g_{\kil}$ on the compact Lie group $\SO(n)$, such that any generalized Ricci flow solution starting at a diagonal metric in $U$ converges to $g_{\kil}$.
\end{theorem}

This implies that near the Killing metric of $\SO(n)$, the conjecture given in \cite[Conjecture 4.14]{GrcStr} holds, i.e.,  for any initial condition $(g,H_0)$ close enough to $(g_{\kil},H_0)$, where $g$ is a left-invariant metric on $\SO(n)$ diagonal with respect to $g_{\kil}$ and $H_0$ is the Cartan $3$-form, the generalized Ricci flow exists on $[0,\infty)$ and converges to the Bismut Ricci flat structure $(g_{\kil},H_0)$.

\vs \noindent \textit {Acknowledgements.} I wish to express my deep gratitude to my Ph.D. advisor Dr. Jorge Lauret for his continued guidance and to the ICTP support during my research visit to Trieste.  Also, I would like to thank to Dr. Dami\'an Fern\'andez for very helpful conversations.

\section{Preliminaries}
\subsection{Aligned homogeneous spaces}\label{preli1}

The known results on compact homogeneous spaces $G/K$ differs substantially between the cases of $G$ simple and non-simple.
One potential reason for this could be that the isotropy representation of $G/K$ is rarely  multiplicity-free  when $G$ is non-simple.
The class of homogeneous spaces with the richest third cohomology were studied in \cite{H3}, they are called \textit{aligned} homogeneous spaces due to their special properties concerning the decomposition in irreducibles of $G$ and $K$ and their Killing constants. We will give an overview of the definition and properties when $G$ has only two simple factors ($s=2$ in \cite{BRFs}).

Let $M=G/K$ be a homogeneous space, where $G$ is a compact, connected, semisimple with two simple factors Lie group and $K$ is a connected closed subgroup. We fix the following decomposition for the Lie algebra of $G$ and $K$:
\begin{equation}\label{decss}
\ggo=\ggo_1\oplus\ggo_2, \qquad \kg=\kg_0\oplus\kg_1\oplus\dots\oplus\kg_t,
\end{equation}where $\ggo_i$'s and $\kg_j$'s are simple ideals of $\ggo$ and $\kg$, respectively, and $\kg_0$ is the center of $\kg$. We call $\pi_i:\ggo\rightarrow\ggo_i$ the usual projections and we set $Z_i:=\pi_i(Z)$ for any $Z\in\ggo$, $i=1,2$.  The Killing form of any Lie algebra $\hg$ will always be denoted by $\kil_\hg$.

\begin{definition}\label{alig-def-2}
A homogeneous space $G/K$ as above is said to be {\it aligned} if there exist $c_1,c_2>0$ such that:
\begin{enumerate}[{label=(\roman*)}]
\item The Killing constants, defined by
\begin{equation*}
\kil_{\pi_i(\kg_j)} = a_{ij}\kil_{\ggo_i}|_{\pi_i(\kg_j)\times\pi_i(\kg_j)},
\end{equation*}
satisfy the following alignment property:
$$
(a_{1j},a_{2j}) = \lambda_j(c_1,c_2) \quad\mbox{for some}\quad \lambda_j>0, \qquad\forall j=1,\dots,t.
$$
\item There exist an inner product $\ip$ on $\kg_0$ such that
\begin{equation}\label{al2}
\kil_{\ggo_i}(Z_i,W_i) = -\tfrac{1}{c_i}\langle Z,W\rangle, \qquad\forall Z,W\in\kg_0, \quad i=1,2.
\end{equation}
\item  $\frac{1}{c_1}+\frac{1}{c_2}=1$.
\end{enumerate}
\end{definition}
The ideals $\kg_j$'s are therefore uniformly embedded on each $\ggo_i$ in some sense. From the definition $G/K$ is automatically aligned if $\kg$ is simple or one-dimensional and the following properties hold:
\begin{enumerate}[{label=(\roman*)}]\label{itemm}
\item $\pi_i(\kg)\simeq\kg$ for $i=1,2$.
\item The Killing form of $\kg_j$ is given by
$\kil_{\kg_j}=\lambda_j\kil_{\ggo}|_{\kg_j\times\kg_j}, \qquad\forall j=1,\dots,t.$
\end{enumerate}

From now on, given homogeneous spaces $M_i=G_i/K$, $i=1,2$, such that each $G_i$ is simple and their Killing constants satisfy $\kil_{\kg}=a_i\kil_{\ggo_i}|_{\kg}$ for $i=1,2$, we consider the homogeneous space $M=G_1\times G_2 / \Delta K$, which is an aligned homogeneous space with
\begin{equation}\label{condition}
c_1=\tfrac{a_1+a_2}{a_2}\text{,} \quad c_2=\tfrac{a_1+a_2}{a_1} \quad \text{ and } \quad \lambda:=\lambda_1=\dots=\lambda_t=\tfrac{a_1a_2}{a_1+a_2}.
\end{equation}
If $a_1\leq a_2$, then $ 1 < c_1 \leq 2 \leq c_2$.

Given $G:=G_1\times G_2$, we consider the reductive decomposition $\ggo=\kg\oplus\pg$, which is orthogonal with respect to $g_{\kil}$, the Killing metric of $G$. We fix the $G$-invariant metric on $M$ called \textit{standard} given by
\begin{equation*}
g_{\kil}=(-\Bb_{\ggo_1})|_{\pg\times\pg}+(-\Bb_{\ggo_2})|_{\pg\times\pg},
\end{equation*}
as a background metric. Note that we denote by $g_{\kil}$ both, the bi-invariant metric on the Lie group $G$ and the $G$-invariant metric on $M$.

Consider the $g_{\kil}$-orthogonal Ad($K$)-invariant decomposition $$\pg=\pg_1\oplus\pg_2\oplus\pg_3,$$
where $\pg_i$ is equivalent to the isotropy representation of the homogeneous space $M_i=G_i/\pi_i(K)$ for $i=1,2$ and
\begin{equation}\label{p3}
  \pg_3:=\left\{\bar{Z}=\left(Z_1,-\tfrac{c_2}{c_1}Z_2\right) :  Z\in \kg\right\}
\end{equation}
is equivalent to the adjoint representation $\kg$ (see \cite[Proposition 5.1]{H3}).

In order to use some known results assume the following technical property:
\begin{assumption}\label{assum}
None of the irreducible components of $\pg_1,\pg_2$ is equivalent to any of the simple factors of $\kg$ as $\Ad(K)$-representations and either $\zg(\kg)=0$ or the trivial representation is not contained in any of $\pg_1,\pg_2$ (see \cite[Section 6]{H3} for more details on this assumption).
\end{assumption}

\subsection{ Bismut connection and generalized Ricci flow}\label{preli2}
For further information on the subject of this subsection, we refer to the recent book \cite{GrcStr} and the articles \cite{CrtKrs, Grc, Lee, PdsRff1, PdsRff2, RubTpl, Str1, Str2}.

Given a compact Riemannian manifold $(M,g)$ and a $3$-form $H$ on $M$, we call {\it Bismut} the unique metric connection on $M$ with torsion $T$ such that it satisfies the $3$-covariant tensor
$$
g(T_XY,Z):=H(X,Y,Z), \qquad\forall X,Y,Z\in\chi(M),
$$
is the $3$-form $H$ on $M$.  When it holds, this connection $\nabla^B$ is given by
$$
g(\nabla^B_XY,Z)=g(\nabla^g_XY,Z)+\unm H(X,Y,Z), \qquad\forall X,Y,Z\in\chi(M),
$$
where $\nabla^g$ is the Levi Civita connection of $(M,g)$.

If $H$ is closed then the pair $(g,H)$ is called a {\it generalized metric}. A way to make these structures evolve naturally is provided by the Ricci tensor of the Bismut connection giving rise to the evolution equation \eqref{GRF} called {\it generalized Ricci flow} (see \cite{GrcStr, Lee} and references therein).

The fixed points of this flow are the {\it Bismut Ricci flat} (BRF for short) generalized metrics, also called {\it generalized Einstein metrics}, i.e., $(g,H)$ such that
\begin{equation}
\ricci(g)=\unc H_g^2 \quad\mbox{and}\quad \mbox{$H$ is $g$-harmonic}.
\end{equation}

\subsection{ Bismut Ricci flat metrics on aligned homogeneous spaces}\label{preli3}

We review in this section the homogeneous BRF generalized metrics recently found in \cite{Corr}, which is a new version of \cite{BRFs} including a Corrigendum.

According to \cite{H3} a bi-invariant symmetric bilinear form $Q_0$ on $\ggo$ defined by
$$
Q_0=\Bb_{\ggo_1}-(\tfrac{1}{c_1-1})\Bb_{\ggo_2},
$$
defines a $G$-invariant closed $3$-form on M denoted $H_0$ and given by
{\small{$$H_0(X,Y,Z) := Q_0([X,Y],Z) + Q_0([X,Y]_\kg,Z) - Q_0([X,Z]_\kg,Y) + Q_0([Y,Z]_\kg,X)  \quad  \text{for all\ }  X,Y,Z \in \pg,$$}}
which is $g$-harmonic for any $G$-invariant metric $g=(x_1,x_2,x_3)_{g_{\kil}}$ of the form:
$$g:=x_1g_{\kil}|_{\pg_1 \times \pg_1} + x_2g_{\kil}|_{\pg_2 \times \pg_2}+x_3g_{\kil}|_{\pg_3 \times \pg_3}.$$

Recently in \cite[Theorem A.2]{Corr} it was proved the existence of a $G$-invariant BRF generalized metric on any homogeneous space $M=G/K$ where $G$ has two simple factors and the Assumption \ref{assum} holds. In terms of the standard metric as a background one the result is the following (see \cite[Remark A.3]{Corr}).

\begin{theorem}{\upshape\cite[Theorem A.2]{Corr}}\label{theoLW}
  Let $M=G/K$ be an aligned homogeneous space with $s=2$ such that Assumption \ref{assum} holds.
\begin{enumerate}[{label=(\roman*)}]
\item The $G$-invariant generalized metric $(g_0,H_{0})$ defined by
\begin{align}
g_0:=\left(1,\tfrac{1}{c_1-1},\tfrac{c_1}{c_1-1}\right)_{g_{\kil}} \label{gthm}
\end{align}
is Bismut Ricci flat.
\item This is the only $G$-invariant Bismut Ricci flat generalized metrics on $M=G/K$ up to scaling of the form $(g=(x_1,x_2,x_3)_{g_{\kil}},H_0)$.
\end{enumerate}
\end{theorem}

As these metrics are precisely the fixed points of the generalized Ricci flow, we are going to use the ingredients involved in the proof of their main theorem. We consider the homogeneous space $M_i=G_i/K$ for $i=1,2$, and its $\kil_{\ggo_i}$-orthogonal reductive decomposition $\ggo_i=\kg\oplus \pg_i$. For each $i=1,2$ we endowed it with the standard metric, which we denote by $g_{\kil}^i$ (i.e., $g_{\kil}^i=-\kil_{\ggo_i}|_{\pg_i\times\pg_i}$). For what follows, we called $\cas_{\chi_{i}}:\pg_i\rightarrow \pg_i$, the \textit{Casimir operator} of the isotropy representation $$\chi_i:\kg\rightarrow\End(\pg_i)$$
of $M_i$ with respect to $-\kil_{\ggo_i}|_{\kg\times\kg}$, and we fix this notation:
$$
A_3:=-\tfrac{c_2}{c_1}, \qquad B_3:=\tfrac{1}{c_1}+A_3^2\tfrac{1}{c_2}, \qquad
B_4:=\tfrac{1}{c_1}+\tfrac{1}{c_2}=1.
$$

The following proposition is proved in \cite[Proposition 3.2]{BRFs} and give us the formula for the Ricci operator when condition \eqref{condition} holds ($\lambda:=\lambda_1=\dots=\lambda_t$). Note that the added hyphothesis only changes $(iii)$ from the original proposition.
\begin{proposition}{\upshape\cite[Proposition 3.2]{BRFs}}\label{ricggals2}
If $s=2$ and \eqref{condition} holds, then the Ricci operator of the metric $g=(x_1,x_2,x_3)_{g_{\kil}}$ is given as follows:
\begin{enumerate}[{label=(\roman*)}]
\item $\Ricci(g)|_{\pg_1} =  \tfrac{1}{4x_1}I_{\pg_1}
+ \tfrac{1}{2x_1}  \left(1 - \tfrac{x_{3}}{x_1c_1B_{3}}\right) \cas_{\chi_{1}}$.
\item[ ]
\item $\Ricci(g)|_{\pg_2} =  \tfrac{1}{4x_2}I_{\pg_2}
+ \tfrac{1}{2x_2}\left(1 - \tfrac{x_{3}}{x_2c_2B_{3}} A_{3}^2\right) \cas_{\chi_{2}} $.
\item[ ]
\item $\Ricci(g)|_{\pg_3}= rI_{\pg_3}$, where
\begin{align*}
r :=& \tfrac{\lambda}{4x_{3}B_3}\left(\tfrac{2x_1^2-x_{3}^2}{x_1^2}
+\tfrac{(2x_2^2-x_{3}^2)A_3^2}{x_2^2}
-\tfrac{1+A_3}{B_3}\left(\tfrac{1}{c_1}+\tfrac{1}{c_2}A_3^3\right)\right) \\
&+\tfrac{1}{4x_{3}B_3}\left(2\left(\tfrac{1}{c_1}+\tfrac{1}{c_2}A_3^2\right)
-\tfrac{2x_1^2-x_{3}^2}{x_1^2c_1}
- \tfrac{(2x_2^2-x_{3}^2)A_3^2}{x_2^2c_2}\right).
\end{align*}
\item[ ]
\item $g(\Ricci(g)\pg_i,\pg_j)=0 \quad \text{for all } i\neq j$.
\end{enumerate}

\end{proposition}

The following results give us the formula of the symmetric bilinear form $(H_0)_g^2$ used in the definition of the generalized Ricci Flow. In general, a bi-invariant symmetric bilinear form  $Q$ on $\ggo$,
$$
Q=y_1\Bb_{\ggo_1}+y_2\Bb_{\ggo_2} \text{ such that } \tfrac{y_1}{c_1}+\tfrac{y_2}{c_2}=0,
$$
defines a $G$-invariant closed $3$-form on M given by
{\small{$$H_Q(X,Y,Z) := Q([X,Y],Z) + Q([X,Y]_\kg,Z) - Q([X,Z]_\kg,Y) + Q([Y,Z]_\kg,X) , \quad  \text{for all\ }  X,Y,Z \in \pg.$$}}
\begin{proposition}{\upshape\cite[Proposition 4.2]{BRFs}}\label{HQ2s}
For any $X\in\pg_k$, $k=1,2$,
$$
(H_Q)_g^2(X,X) = g_{\kil}\left(\left(\left(\tfrac{2S_k}{x_kc_k}-\tfrac{2y_k^2}{x_k^2}\right)\cas_{\chi_k}+\tfrac{y_k^2}{x_k^2}I_{\pg_k}\right)X,X\right),
$$
where
$$
C_3=\tfrac{y_1}{c_1}+A_3\tfrac{y_2}{c_2},\qquad
S_1= \tfrac{1}{x_3B_3}\left(y_1 + \tfrac{C_3}{B_4}\right)^2, \qquad
S_2 = \tfrac{1}{x_3B_3}\left(A_3y_2+ \tfrac{C_3}{B_4}\right)^2.
$$
\end{proposition}

\begin{proposition}{\upshape\cite[Proposition A.1]{Corr}}\label{HQ2-s2}
If $\bar{Z}\in\pg_3$ (see \eqref{p3}), with $g_{\kil}(\bar{Z},\bar{Z})=1$, then
\begin{align*}
(H_Q)_{g}^2(\bar{Z},\bar{Z})
= & \tfrac{1}{x_1^2B_3} \left(y_1 + \tfrac{C_3}{B_{4}} \right)^2 \tfrac{1-c_1\lambda}{c_1}
 +\tfrac{1}{x_2^2B_3} \left(y_2A_3 + \tfrac{C_3}{B_{4}}\right)^2 \tfrac{1-c_2\lambda}{c_2} \\
& +\tfrac{\lambda}{x_3^2B_3^3}\left(\tfrac{y_1}{c_1}+A_3^3\tfrac{y_2}{c_2} + \tfrac{3C_3}{B_{4}} \left(\tfrac{1}{c_1}+A_3^2\tfrac{1}{c_2}\right)\right)^2.
\end{align*}
\end{proposition}

\section{Generalized Ricci flow on aligned spaces}

The aim of this paper is to study the generalized Ricci flow on homogeneous spaces of the form $M=G_1\times G_2 / K$, such that each $G_i$ is a simple Lie group, $K$ is a closed subgroup of them and their Killing constants satisfy $\kil_{\kg}=a_i\kil_{\ggo_i}|_{\kg}$ for $i=1,2$. As in Section \ref{preli1}, $M$ is an aligned homogeneous space satisfying \eqref{condition}, note that $a_i=\lambda c_i$, for $i=1,2$. We also assume that each homogeneous space $M_i:=G_i/K$ has its standard metric $g_{\kil}^i$ Einstein for $i=1,2$. In that sense we look for invariant solutions $(g(t),H(t))$ to the equations given in \eqref{GRF}.

We set $x_1:=x_1(t)$, $x_2:=x_2(t)$, $x_3:=x_3(t)$ smooth positive functions and define
\begin{equation*}\label{def g}
 g(t):=x_1g_{\kil}|_{\pg_1 \times \pg_1} + x_2g_{\kil}|_{\pg_2 \times \pg_2}+x_3g_{\kil}|_{\pg_3 \times \pg_3}.
 \end{equation*}

As every scalar multiple of $H_0$ is $g(t)$-harmonic for every $t\geq0$, the second equation of the flow in \eqref{GRF} vanishes and $H(t)\equiv H_0$, therefore we only need to work with the first one, starting from some generalized metric of the form $(g(0),H_0)$, that is
\begin{equation}\label{GRF1}
\tfrac{\partial}{\partial t}g(t)=-2\Ricci(g(t))+\tfrac{1}{2}(H_0)^2_{g(t)}.
\end{equation}

Note that, as we define it, $g(t)$ is a diagonal metric for all $t \geq 0$. In the next lemma we see that, under certain conditions, the set of diagonal metrics is invariant under the flow.

\begin{lemma}\label{inv_}
Let $M=G_1\times G_2 / K$ be an aligned homogeneous space satisfying \eqref{condition} and assume that $(M_i =G_i/K, g_{\kil}^i)$ is  Einstein, where $g_{\kil}^i$ is the standard metric on $M_i$, then the set of diagonal metrics with respect to the decomposition $\pg=\pg_1\oplus\pg_2\oplus\pg_3$ is invariant under the generalized Ricci flow.
\end{lemma}
\begin{proof}
 As $(M_i, g_{\kil}^i)$ is Einstein, there exist constants $ \kappa_i \in (0,\tfrac{1}{2}]$ such that $\cas_{\chi_{i}} = \kappa_i \id_{\pg_i}$ for $i=1,2$. Therefore, from Propositions \ref{ricggals2}, \ref{HQ2s} and \ref{HQ2-s2}, we see that $\tfrac{d}{dt}g(t)$ is tangent to the space of diagonal metrics, because all the operators involved are multiples of the identity in each $\pg_j$ for $j=1,2,3$.
\end{proof}

\begin{proposition}\label{floww}
Let $M=G_1\times G_2/ K$ be as in Lemma \ref{inv_} such that $\cas_{\chi_{i}} = \kappa_i \id_{\pg_i}$ for $i=1,2$. Fix the standard metric of $M$, $g_{\kil}$, as a background metric, then the generalized Ricci flow for metrics of the form $g=(x_1,x_2,x_3)_{g_{\kil}}$ is given by the following system of ODE's:
\begin{align}\label{flow z11}
\left\{\begin{aligned}
  x_1'(t)= &\tfrac{2 \kappa_1 x_1 x_3^2 + c_1 x_3 \left(-1 + x_1^2 + 2 \kappa_1 (1 + x_1^2 - 2 x_1 x_3)\right) + c_1^2 \left(x_3 - x_1^2 x_3 - 2 \kappa_1 \left(x_3 + x_1^2 x_3 - x_1 (1 + x_3^2)\right)\right)}{2 (c_1-1) c_1 x_1^2 x_3},\\
  x_2'(t)= &-\tfrac{c_1^3 (1 + 2 \kappa_2) x_2^2 x_3 - 2 \kappa_2 x_2 x_3^2 +
 c_1 x_3 \left(-1 + x_2^2 + 2 \kappa_2 (1 + x_2^2 + 2 x_2 x_3)\right) -
 2 c_1^2 x_2 \left(x_2 x_3 + \kappa_2 (1 + 2 x_2 x_3 + x_3^2)\right)}{2 (c_1-1)^2 c_1 x_2^2 x_3},\\
  x_3'(t)=&\tfrac{\left(x_3^2 - 2 c_1 x_3^2 + c_1^2 (x_3^2-1)\right)\left(-c_1^2 (1 + 2\lambda) x_2^2 x_3^2 + (x_1^2 - x_2^2) x_3^2 + c_1 \left((\lambda-1) x_1^2 + (2 +\lambda) x_2^2\right) x_3^2 + c_1^3\lambda x_2^2 (x_3^2-x_1^2)\right)}{2 (c_1-1)^3 c_1 x_1^2 x_2^2 x_3^2}.\\
\end{aligned} \right.
\end{align}
\end{proposition}
\begin{proof}
From Definition \ref{alig-def-2} $(iii)$, $c_2=\tfrac{c_1}{c_1-1}$. Hence, as in \cite[Section 5]{BRFs}, we have that
$A_3=-\tfrac{1}{c_1-1}\text{ and } B_3=\tfrac{1}{c_1-1}$. Note that for $j=1,2,3$,
$$g(\Ricci(g)\pg_j,\pg_j)=x_jg_{\kil}(\Ricci(g)\pg_j,\pg_j),$$
and $\Ricci(g)|_{\pg_j}$ is a multiple of the identity in $\pg_j$ given by Proposition \ref{ricggals2}.

Given $Q_0$ with $y_1=1$ and $y_2=-\tfrac{1}{c_1-1}$, the formula of $(H_0)^2_{g}$ follows from Propositions \ref{HQ2s} and \ref{HQ2-s2} with
$$ C_3=\tfrac{1}{c_1-1}, \quad S_1=\tfrac{c_1^2}{(c_1-1)x_3}, \text{ and } \ S_2=\tfrac{c_1^2}{(c_1-1)^3x_3}. $$
Therefore, replacing these on \eqref{GRF1} the proposition holds.
\end{proof}

Note that we can write these equations using the usual notation of nonlinear systems of differential equations
\begin{equation} \label{diff}
x'(t)=f(x),
\end{equation}
where $x:=(x_1,x_2,x_3)$ and $f: \RR^3_{>0}\rightarrow \RR^3$ is given by the equations in \eqref{flow z11}.

According to Theorem \ref{theoLW} we know that $(g_0,H_0)$ defined by \eqref{gthm} is a BRF generalized metric, which means a fix or an equilibrium point of the flow. If we set $x_0:=(1,\tfrac{1}{c_1-1},\tfrac{c_1}{c_1-1})$ then $f(x_0)=0$ and the local behaviour of \eqref{flow z11} is qualitatively determined by the behaviour of the linear system $x'=Ax$ near the origen, where $A=Df(x_0)$, the derivative of $f$ at $x_0$.

To begin with the study of generalized Ricci flow on this class of aligned homogeneous spaces we define some invariant subspaces and show some plots. Regardless of whether $K$ is abelian $(\lambda=0)$ or not, we see that:
\begin{itemize}
  \item  The plane given by $x_3=\tfrac{c_1}{c_1-1}$ is invariant by the flow.
  \item  Inside the plane, the lines defined by $x_1=1$ and $x_2=\tfrac{1}{c_1-1}$ are invariant, and
  \item when $\kappa_1=\kappa_2$, the plane $x_1=(c_1-1)x_2$ and the line in it given by $x_3=\tfrac{c_1}{c_1-1}$ are also invariant.
\end{itemize}

Just to ilustrate the flow in those invariant planes we consider $M=\SU(7)\times\SO(8)/\SO(7)$, a homogeneous space of dimension $55$ such that $c_1=\tfrac{10}{7}$, $\kappa_1=\kappa_2=\tfrac{1}{2}$ and $\lambda=\tfrac{1}{4}$. In this case, the BRF metric is $x_0=(1,\tfrac{7}{3},\tfrac{10}{3})$ and the plots of the invariant planes of the flow are the following.

\break
\begin{paracol}{2}
\begin{figure}[h!]
  \centering
  \includegraphics[width=6cm]{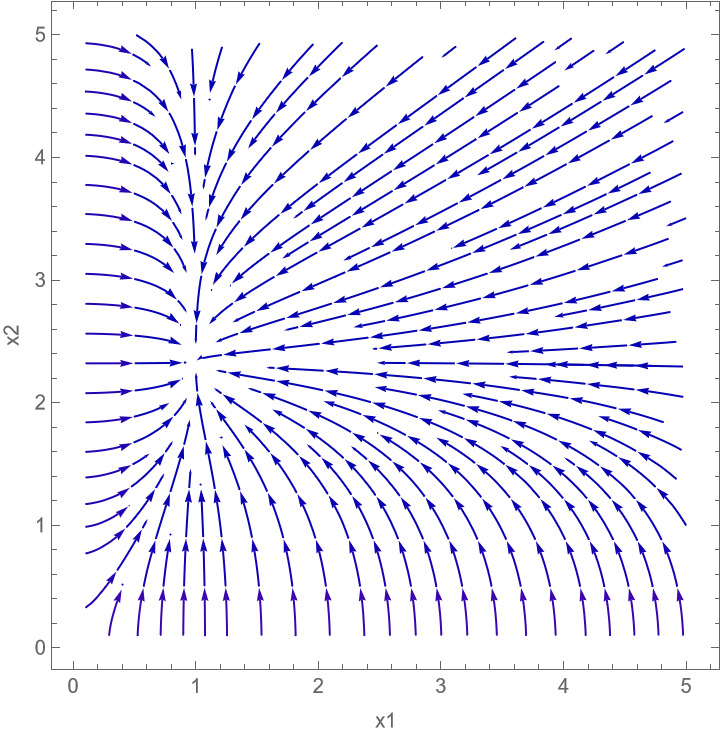}
  \caption{Flow in the invariant plane $x_3=\tfrac{10}{3}$.}
\end{figure}
\switchcolumn
  \begin{figure}[h!]
  \centering
  \includegraphics[width=6cm]{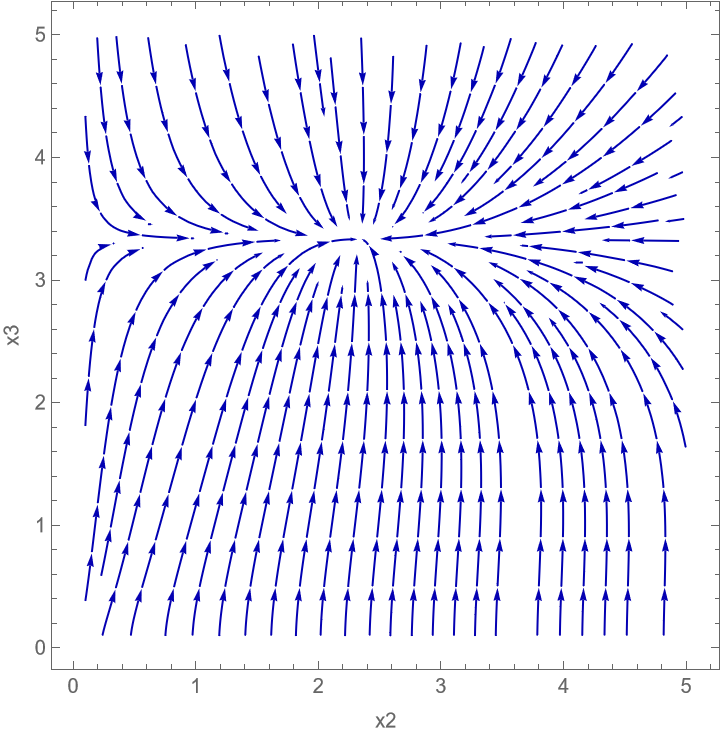}
  \caption{Flow in the invariant plane $x_1=\tfrac{3}{7}x_2$.}
\end{figure}
\end{paracol}

For the remainder of this section, we recall some definitions of the local theory of nonlinear systems. Consider the system given in \eqref{diff},
\begin{definition}
An equilibrium point $x_0$ (i.e., $f(x_0)=0$) is called \textit{hyperbolic} if none of the eigenvalues of the matrix $Df(x_0)$ have zero real part.
\end{definition}
\begin{definition}
  Let $\phi_t$ denote the flow of the differential equation in \eqref{diff} defined for all $t\in \RR$. An equilibrium point $x_0$ is \textit{stable} if for all $\varepsilon>0$ there exists a $\delta>0$ such that for all $x\in N_{\delta}(x_0)$ and $t\geq 0$ we have
  $$ \phi_t(x) \in N_{\varepsilon}(x_0),$$
  where $N_{\alpha}(x_0)$ is the open ball of positive radius $\alpha$ centered at $x_0$.

  $x_0$ is \textit{asymptotically stable} if it is stable and there exists  $\delta>0$ such that for all $x \in N_{\delta}(x_0)$ we have
  $$\lim_{t\rightarrow \infty}\phi_t(x)=x_0.$$
\end{definition}

Because the stability of an equilibrium point is a local property it is reasonable to expect that it will be the same as the stability at the origin of the linear system $x'(t) = Df(x_0)x$. This expectation is not always realized, but it is for hyperbolic equilibrium points.

The next result gives us a better understanding of the local behaviour of the flow near the BRF generalized metric.

\begin{theorem}\label{stable z11}
  The metric $g_0=(1,\tfrac{1}{c_1-1},\tfrac{c_1}{c_1-1})_{g_{\kil}}$ is asymptotically stable for the dynamical system \eqref{flow z11}.
\end{theorem}

\begin{proof}
From Theorem \ref{theoLW} we know that $(g_0,H_0)$ is a BRF generalized metric, which means an equilibrium point of the generalized Ricci flow.

If $f:=\left(f_1(x_1,x_2,x_3),f_2(x_1,x_2,x_3),f_3(x_1,x_2,x_3)\right)$ satisfies
$$\left\{ \begin{aligned}
  x_1'(t) &= f_1(x_1,x_2,x_3), \\
  x_2'(t) &= f_2(x_1,x_2,x_3) ,\\
  x_3'(t) &= f_3(x_1,x_2,x_3),
\end{aligned}\right.$$
as in \eqref{flow z11}, then $f(g_0)=0$ and its differential at $g_0$ is given by:

$$Df(g_0)=\begin{pmatrix}
              -1 & 0 & 0 \\
              0 & 1-c_1 & 0 \\
              0 & 0 & c_1 (-1 + \lambda)
            \end{pmatrix}.$$

As all the eigenvalues of the matrix are real and negative, the equilibrium point $g_0=(1,\tfrac{1}{c_1-1},\tfrac{c_1}{c_1-1})_{g_{\kil}}$ is asymptotically stable by \cite[Section 2.9]{Per}.
\end{proof}

%

\section{Global stability}

In the previous section we study the local behaviour of the non-linear dynamical system \eqref{flow z11}, the study of its global behaviour is substancially harder due to the possibility of caos.
\begin{definition}
An equilibrium point $x_0$ of a nonlinear system of differential equations $x'(t) = f(x)$ is \textit{globally stable} if it is stable and globally attractive, which means
$$\lim_{t\rightarrow \infty}\phi_t(x)=x_0 \text{ for all }  x \in D,$$
where $D$ is the domain of $f$.
\end{definition}
This definition says that an equilibrium point is globally stable if the set of points in the space that are asymptotic to it is the whole space. The Lyapunov stability theorems provide sufficient conditions for this type of stability as well as asymptotic stability, this approach is based on finding a scalar function of a state that satisfies certain properties. Namely, this function has to be continuously differentiable and positive definite. Besides,  if the first derivative of this function with respect to time is negative semidefinite along the state trajectories, then we can conclude that the equilibrium point is stable. Further, if the first derivative along every state trajectories is negative definite then we can conclude that the equilibrium point is globally stable.

As there is no universal method for creating Lyapunov functions for ordinary differential equations, this problem is far from trivial.

\begin{theorem}\cite[Section 2.9]{Per}
Let  $E$ be an open subset of $\RR^n$ cointaining $x_0$, the equilibrium point. Suppose that $f\in C^1(E)$ and that $f(x_0)=0$. Suppose further that there exits a real valued function $V \in C^1(E)$ (called \textit{Lyapunov function}) satisfying $V(x_0)=0$ and $V(x)>0$ if $x\neq x_0$. Then,
\begin{itemize}
  \item [a)] if $V'(x)\leq 0$ for all $x\in E$, $x_0$ is stable.
  \item [b)] If $V'(x)< 0$ for all $x\in E \setminus \{x_0\}$, $x_0$ is asymptotically stable.
  \item [c)] If $V'(x)> 0$ for all $x\in E \setminus \{x_0\}$, $x_0$ is unstable.
\end{itemize}
\end{theorem}

Given the space $H/K$, where $H$ is a simple Lie group and $K\subseteq H$, one can consider the homogeneous space $M = H\times H/\Delta K$, such that $G_1 = G_2 = H$ and $c_1 = 2$. We are going to prove global stability for this special case when hyphothesis used before hold.

\begin{theorem}\label{global}
  Let $M=H\times H/\Delta K$,  ($c_1=2$) a homogeneous space such that $(H/K, g_{\kil})$ is  Einstein  (i.e. $\cas_{\chi} = \kappa \id_{\pg}$ for a constant $\kappa \in (0,\tfrac{1}{2}]$) and  $\kil_{\kg}=2\lambda\kil_{\hg}|_{\kg}$. Then the Bismut Ricci flat metric $g_0:=(1,1,2)_{g_{\kil}}$ is globally stable.
\end{theorem}
\begin{proof}
  Consider the function,
  $$V(x_1,x_2,x_3):= \frac{\tfrac{\lambda}{10}(x_1-1)^2+\tfrac{\lambda}{10}(x_2-1)^2+(x_3-2)^2}{2},$$
we are going to prove that this is a Lyapunov function for the dynamical system
\eqref{flow z11} taking $c_1=2$ and $\kappa:=\kappa_1=\kappa_2$, i.e.,

  \begin{align}\label{flow c12}
\left\{\begin{aligned}
  x_1'(t)= &\tfrac{x_3 - x_1^2 x_3 + \kappa \left(-2 x_3 - 2 x_1^2 x_3 + x_1 (4 + x_3^2)\right)}{2 x_1^2 x_3},\\
  x_2'(t)= &\tfrac{x_3 - x_2^2 x_3 + \kappa \left(-2 x_3 - 2 x_2^2 x_3 + x_2 (4 + x_3^2)\right)}{2 x_2^2 x_3},\\
  x_3'(t)=&-\tfrac{(x_3^2-4) \left((x_1^2 + x_2^2) x_3^2 -
    2 \lambda \left(x_2^2 x_3^2 + x_1^2 (-4 x_2^2 + x_3^2)\right)\right)}{4 x1^2 x2^2 x3^2}.\\
\end{aligned} \right.
\end{align}

    Set $E:=\{(x_1,x_2,x_3)\in \RR^3 \  / \  x_1,x_2,x_3>0\}$, it is obvious that $V(1,1,2)=0$  and  $V$ is clearly positive for all other $(x_1,x_2,x_3)\in E$, then $V$ is a Lyapunov function for the dynamical system \eqref{flow c12} if its derivative $F$ is negative definite on $E \setminus \{(1,1,2)\}$, where

  $$\begin{aligned}
    F(x_1,x_2, x_3):= & \frac{\lambda}{10}(x_1-1)x_1'+\frac{\lambda}{10}(x_2-1)x_2'+(x_3-2)x_3'= \\
   = &-\frac{1}{4 x_1^2 x_2^2 x_3^2}\left(64 \lambda x_1^2 x_2^2 \right.\\
   &+  \tfrac{8}{10}\lambda \kappa x_1^2 x_2 x_3 +  \tfrac{8}{10} \lambda\kappa x_1 x_2^2 x_3 -  \tfrac{8}{10}\lambda \kappa x_1^2 x_2^2 x_3 -  \tfrac{8}{10}\lambda \kappa x_1^2 x_2^2 x_3 - 32  \lambda x_1^2 x_2^2 x_3 \\
   &+  \tfrac{2}{10}\lambda x_1^2 x_3^2 + 8 x_1^2 x_3^2 -  \tfrac{4}{10}\lambda \kappa x_1^2 x_3^2 - 16 \lambda x_1^2 x_3^2 -  \tfrac{2}{10}\lambda x_1^2 x_2 x_3^2 +   \tfrac{4}{10}\lambda \kappa x_1^2 x_2 x_3^2 \\
   &+  \tfrac{2}{10}\lambda x_2^2 x_3^2 + 8 x_2^2 x_3^2 -  \tfrac{4}{10}\lambda \kappa x_2^2 x_3^2 - 16\lambda x_2^2 x_3^2 -  \tfrac{2}{10}\lambda x_1 x_2^2 x_3^2 +  \tfrac{4}{10}\lambda \kappa x_1 x_2^2 x_3^2 \\
   &-  \tfrac{2}{10}\lambda x_1^2 x_2^2 x_3^2 -  \tfrac{2}{10}\lambda x_1^2 x_2^2 x_3^2 -
    \tfrac{4}{10}\lambda \kappa x_1^2 x_2^2 x_3^2 -  \tfrac{4}{10}\lambda \kappa x_1^2 x_2^2 x_3^2 - 16 \lambda x_1^2 x_2^2 x_3^2 \\
   &+  \tfrac{2}{10}\lambda x_1^3 x_2^2 x_3^2 +
    \tfrac{4}{10}\lambda \kappa x_1^3 x_2^2 x_3^2 +  \tfrac{2}{10}\lambda x_1^2 x_2^3 x_3^2 +
    \tfrac{4}{10}\lambda \kappa x_1^2 x_2^3 x_3^2 \\
   & - 4  x_1^2 x_3^3 + 8 \lambda x_1^2 x_3^3 +
    \tfrac{2}{10}\lambda \kappa x_1^2 x_2 x_3^3 - 4 x_2^2 x_3^3 + 8 \lambda x_2^2 x_3^3 +
    \tfrac{2}{10}\lambda \kappa x_1 x_2^2 x_3^3 \\
   &-  \tfrac{2}{10}\lambda \kappa x_1^2 x_2^2 x_3^3 -  \tfrac{2}{10}\lambda \kappa x_1^2 x_2^2 x_3^3 + 8 \lambda x_1^2 x_2^2 x_3^3\\
   & - 2 x_1^2 x_3^4 + 4 \lambda x_1^2 x_3^4 - 2  x_2^2 x_3^4 + 4 \lambda x_2^2 x_3^4 +
   \left. x_1^2 x_3^5 - 2 \lambda x_1^2 x_3^5 +  x_2^2 x_3^5 - 2 \lambda x_2^2 x_3^5\right).
   \end{aligned}$$

  Due to the symmetries of $x_1$ and $x_2$, in order to prove the negative condition of $F$, we define a function $g$ such that,
   $$F(x_1,x_2,x_3)=-\tfrac{1}{4 x_1^2 x_2^2 x_3^2}\left(x_1^2 g(x_2,x_3)+x_2^2 g(x_1,x_3)\right),$$
  where
  $$\begin{aligned}
  g(x,y)&:= 32\lambda x^2 +  \tfrac{8}{10}\lambda \kappa x y -  \tfrac{8}{10}\lambda \kappa  x^2 y - 16 \lambda x^2 y + \tfrac{2}{10} \lambda y^2 + 8 y^2 -  \tfrac{4}{10}\lambda \kappa y^2 - 16 \lambda y^2\\
  &-  \tfrac{2}{10}\lambda x y^2 +  \tfrac{4}{10}\lambda \kappa x y^2 - \tfrac{2}{10}\lambda x^2 y^2 -  \tfrac{4}{10}\lambda \kappa x^2 y^2 - 8 \lambda x^2 y^2 +  \tfrac{2}{10}\lambda x^3 y^2 +  \tfrac{4}{10} \lambda \kappa x^3 y^2 \\
  & - 4y^3 + 8\lambda y^3 +  \tfrac{2}{10}\lambda \kappa x y^3 -
   \tfrac{2}{10}\lambda \kappa x^2 y^3 + 4 \lambda x^2 y^3 - 2  y^4 + 4 \lambda y^4 +  y^5 - 2 \lambda y^5.
  \end{aligned} $$

  Hence, it is enough to prove that $g(x,y)\geq0$ on $\tilde{E}$, where $\tilde{E}:=\{(x,y)\in \RR^2 \ / \ x,y>0\}$, and equality holds only in $(x,y)=(1,2)$. We split the function $g$ considering the terms where $\lambda$ or $\kappa$ are involved,
  $$g(x,y):=\frac{1}{10}\lambda h_1(x,y) +h_2(x,y)+\frac{2}{10}\lambda\kappa h_3(x,y) + 2 \lambda h_4(x,y),$$
  where
  $$\begin{aligned}
  h_1(x,y) &:=2 y^2 - 2  x y^2 - 2  x^2 y^2 + 2  x^3 y^2,\\
  h_2(x,y) &:= 8 y^2 - 4 y^3 - 2 y^4 + y^5,\\
  h_3(x,y) &:= 4 x y - 4 x^2 y - 2 y^2 + 2 x y^2 - 2 x^2 y^2 + 2 x^3 y^2 + x y^3 - x^2 y^3,\\
  h_4(x,y) &:= 16 x^2 - 8 x^2 y - 8 y^2 - 4 x^2 y^2 + 4 y^3 + 2 x^2 y^3 +
  2 y^4 - y^5.
\end{aligned}$$
Now, from its factorization it is clear that $h_1$ and $h_2$ are greater than zero on $\tilde{E} \setminus \{(1,2)\}$,
  $$\begin{aligned}
  h_1(x,y) &:=2  (x-1)^2 (1 + x) y^2,\\
  h_2(x,y) &:= (y-2)^2 y^2 (2 + y).\\
  \end{aligned} $$

The next step is defining the functions $g_1,g_2$ in the following way:

  $$\begin{aligned}
  g_1(x,y) &:=h_4(x,y)+\tfrac{1}{20}h_1(x,y),\\
  g_2(x,y) &:= \tfrac{1}{20}h_3(x,y),\\
\end{aligned}$$
such that
$$g(x,y):=h_2(x,y)+2\lambda g_1(x,y)+4\kappa \lambda g_2(x,y). $$

The proof that $g(x,y)>0$ on $\tilde{E} \setminus \{(1,2)\}$ is going to be by cases, as $h_2$ is already positive, we have to differenciate if $g_1$ and $g_2$ are positive or not on $\tilde{E} \setminus \{(1,2)\}$.

\begin{paracol}{3}
  \begin{figure}[ht]
  \centering
  \includegraphics[width=4.5cm]{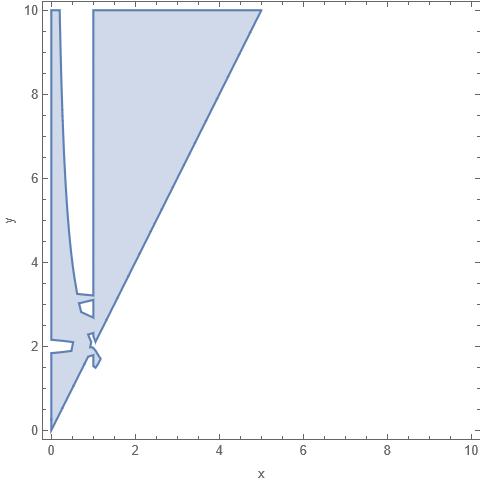}
  \captionsetup{font=small,labelfont=small,width=4cm}
  \caption{Region of Case 1: $g_1<0$ and $g_2<0$.}
  \label{g1neg g2neg}
  \end{figure}
\switchcolumn
  \begin{figure}[ht]
  \centering
  \includegraphics[width=4.5cm]{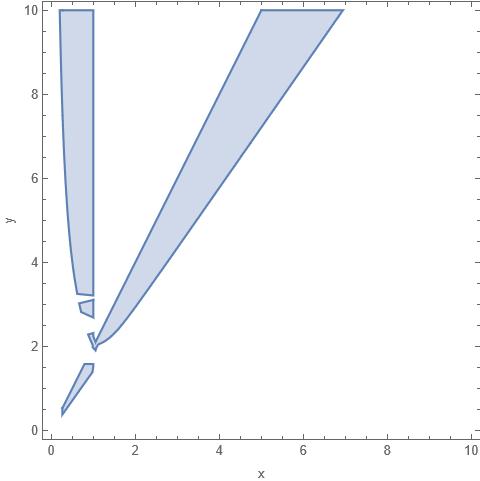}
  \captionsetup{font=small,labelfont=small,width=4cm}
  \caption{Region of Case 2: $g_1<0$  and $g_2>0$. }
  \label{g1neg g2pos}
  \end{figure}
\switchcolumn
  \begin{figure}[ht]
  \centering
  \includegraphics[width=4.5cm]{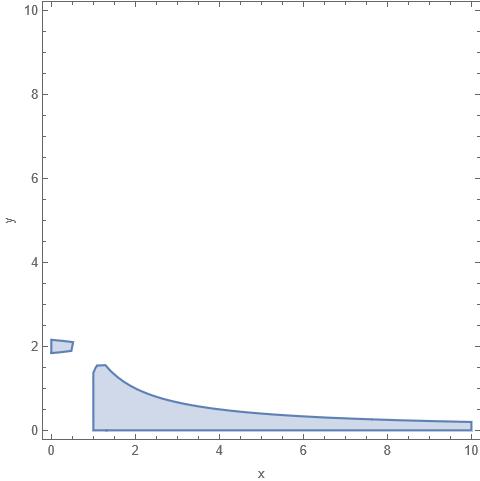}
  \captionsetup{font=small,labelfont=small,width=4cm}
  \caption{Region of Case 3: $g_1>0$  and $g_2<0$. }
  \label{g1pos g2neg}
  \end{figure}
\end{paracol}

\begin{itemize}[itemindent=*, leftmargin=0pt]
  \item \textbf{Case 1: $g_1<0$ and $g_2<0$},

  The $(x,y) \in \tilde{E}$ satisfying these conditions are represented in  Figure \ref{g1neg g2neg}. Note that $g_1<0$ if and only if $h_4<0$.

  Since $g_1<0$ and $2\lambda<1$, we have that $g_1<2\lambda g_1$. In the same way $g_2<4\lambda\kappa g_2$ and the following inequality holds for all $(x,y) \in \tilde{E}$,
  $$ h_2 +g_1+ g_2 \  < \ h_2 +2\lambda g_1+ 4 \lambda \kappa g_2.$$

  Our goal now is to prove that the function $h_2 +g_1+ g_2$ is greater than zero on $\tilde{E}$. It is a cubic function on $x$ such that $p(x)$ defined below is cuadratic,
  \begin{align*}
  (h_2+g_1+g_2)(x,y)&=x\left(\tfrac{1}{20}  y (4 + y^2) +
  x \tfrac{1}{20} (320 - 164 y - 84 y^2 + 39 y^3) + x^2 \tfrac{y^2}{5} \right)\\
  &:=xp(x)
  \end{align*}

  To understand $p(x) \in \RR[y][x]$, we are going to look for its critical point, which is always a minimum due to the fact that the quadratic coefficient $\tfrac{y^2}{5}$ is always positive.

  $$p'(x)=0 \text{ if and only if } x=\bar{x}:=\tfrac{1}{y^2}\left(\tfrac{21}{2}y^2-40+\tfrac{41}{2}y-\tfrac{39}{8}y^3\right).$$

  Since $p(0)=\tfrac{1}{20}  y (4 + y^2)$ is positive for $y>0$, and due to the continuity of $p(x)$, we can assure that if $\bar{x}<0$ then $p(x)\geq 0$ for all $x,y>0$ and hence $h_1+g_1+g_2\geq 0$ for all $x,y>0$ where equality holds only in $(x,y)=(1,2)$.

  Thus, the values of $y$ we have to study are those which make $\bar{x}>0$, which means $y \in I:=\left(\frac{1}{39}\left(81-\sqrt{321}\right),\tfrac{1}{39}\left(81+\sqrt{321}\right)\right)$. For this interval we want that $p(\bar{x})>0$, because if the minimum is greater than zero, the whole function is.
  $$p(\bar{x}):=-\tfrac{(y-2)^2 (25600 - 640 y - 13756 y^2 - 484 y^3 + 1521 y^4)}{
 320 y^2},$$
 and this is positive in the interval we are interested in if and only if $q(y):=(25600 - 640 y - 13756 y^2 - 484 y^3 + 1521 y^4)$ is negative for all $y \in I$. Note that as $q(\tfrac{7}{5})>0, \ q(\tfrac{3}{2})<0 $ and  $q(\tfrac{13}{5})<0, \ q(3)>0$ we can localize the positive roots of $q$ and conclude that there are no roots of $q$ in $I$, (see Figure \ref{qfun}), besides $q(2)=-10240<0$.

  \begin{figure}[h]
  \centering
  \includegraphics[width=5cm]{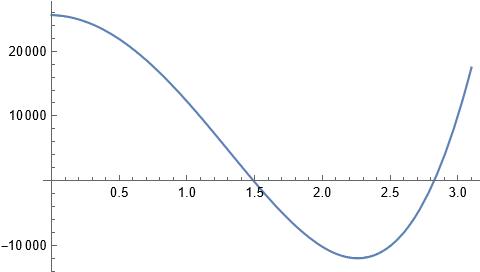}
  \caption{q(y)}
   \label{qfun}
  \end{figure}

  This implies that $p(\bar{x})>0$ for all $y \in I$ and therefore, $h_2+g_1+g_2>0$ on $\tilde{E} \setminus \{(1,2)\}$ as we wanted.

  \item \textbf{Case 2: $g_1<0$  and $g_2>0$},

  The $x,y>0$ satisfying these conditions are ploted in Figure \ref{g1neg g2pos}. As in Case 1, $g_1<0$ implies that $h_4<0$ and therefore $h_4<2\lambda h_4$.
  It is easy to see that
  $$(h_2+h_4)(x,y)=2 x^2 (y-2)^2 (2 + y),$$
  and hence,
  $$0<h_2+h_4< h_2 + 2\lambda h_4 < h_2 + 2\lambda h_4 + 2 \lambda \tfrac{1}{20}h_1,$$
  proving that $g(x,y)>0$ for all $(x,y)\in \tilde{E} \setminus \{(1,2)\}$ in this case.

  \item \textbf{Case 3: $g_1>0$  and $g_2<0$},

  This is the last case we have to analyze, it includes the cases when $x$ tends to infinity and y tends to $0$ as the plot shows, see Figure \ref{g1pos g2neg}.

  Consider $g$ factorized in the following way:
  $$g (x,y)=h_2+2\lambda(g_1+2\kappa g_2).$$

  Since $g_2<0$ in this case, $g_2<2\kappa g_2$, therefore if $g_1+g_2$ is positive this case is proved.

  On the other hand, if $g_1+g_2<0$, since  $g_1>0$, we have the following inequalities:
  $$ g_1+g_2< 2\kappa (g_1 + g_2)< 2\kappa g_1 + 2\kappa g_2 < g_1+ 2\kappa g_2. $$

  If the last expression is still negative for some $x,y>0$ satisfying the hyphothesis of this case, one obtains that:
  $$h_2+g_1+g_2< h_2 + 2\kappa(g_1+g_2)< h_2 + (g_1+2\kappa g_2) < h_2+2\lambda ((g_1+2\kappa g_2) = g(x,y),$$\end{itemize}
  and the proof is complete due to the proof of Case 1. \end{proof}

\section{Generalized Ricci flow on the Lie group $\SO(n)$}\label{so}

\subsection{Set up for any compact Lie group}

Before introducing the main problem of this section, we go over some known facts for compact Lie groups, see \cite[Section 6]{BRFs}.

Let $M^d=G$ be a compact semisimple connected Lie group and $\ggo$ its Lie algebra. By \cite[Chapter V]{Bre}, we know that every class in the set of all closed $3$-forms of $G$, $H^3(G)$ has a unique bi-invariant representative called \textit{Cartan $3$-forms} of the form:
$$\overline{Q}(X,Y,Z):=Q([X,Y],Z), \quad \forall X,Y,Z \in \ggo.$$

We note that if $G$ is simple, then $H^3(G)=\RR\left[H_{\kil}\right]$ where $H_{\kil}:=\overline{\kil_{\ggo}}$, and $\kil_{\ggo}$ is the Killing form of $\ggo$. It is well known that Cartan $3$-forms are harmonic with respect to any bi-invariant metric on $G$.

Using \cite[Section 3]{H3} we can consider a bi-invariant metric $g_b$ on G, a $g_b$-orthonormal basis $\left\{e_1, \dots e_d\right\}$ and a left-invariant metric $g$ on $G$, such that $g(e_i,e_j)=x_i\delta_{ij}$, we denote it by $g=(x_1,\dots,x_d)$. The ordered basis $\left\{e_1, \dots e_d\right\}$ determines structural constants given by
$$c_{ij}^{k}:=g_b([e_i,e_j],e_k),$$
and using \cite[Corollary 3.2 (ii)]{H3}:
\begin{equation}\label{harmo}
 H_{\kil} \text{ is } g\text{-harmonic if and only if,} \sum_{1\leq i,j \leq n}\frac{c_{ij}^kc_{ij}^l}{x_ix_j}=0 \quad \forall k,l \text{ such that } x_k\neq x_l.
\end{equation}

Therefore, considering the generalized Ricci flow, if \eqref{harmo} holds, then $H$ will remain fix along any solution and the flow will be given only by the equation
\begin{equation}\label{flow}
\tfrac{\partial}{\partial t}g(t)=-2\Ricci(g(t))+\tfrac{1}{2}(H(t))^2_{g(t)}.
\end{equation}

We are going to use the formulas calculated in \cite[Section 6]{BRFs}, which are the following for $H_b:=g_b([\cdot,\cdot],\cdot)$:

\begin{equation}\label{Hk2}
(H_b)_g^2(e_k,e_l) = \sum_{i,j} \tfrac{1}{x_ix_j} g_b([e_k,e_i],e_j)g_b([e_l,e_i],e_j)
= \sum_{i,j} \tfrac{c_{ij}^kc_{ij}^l}{x_ix_j}, \qquad\forall k,l.
\end{equation}
Concerning the Ricci curvature, it is well known that
\begin{equation}\label{Rc}
\ricci(g)(e_k,e_l) = \unm\sum_{i,j} c_{ij}^kc_{ij}^l-\unc\sum_{i,j} c_{ij}^kc_{ij}^l\tfrac{x_i^2+x_j^2-x_kx_l}{x_ix_j}, \qquad\forall k, l.
\end{equation}

\subsection{Case $\SO(n)$}
For this section we consider $G=\SO(n)$, the usual basis of its Lie algebra is $\beta=\{e_{rs}:=E_{rs}-E_{sr}\}$ which is \textit{nice}, meaning that it satisfies  $c_{ij}^kc_{ij}^l=0$ for all $k\neq l$.
For computations we enumerate the elements of the basis such that $\beta:=\{e_1,e_2,\dots, e_{\tfrac{n(n-1)}{2}}\}$, this basis is orthogonal with respect to $g_{\kil}$, the Killing metric of $\SO(n)$, and we can write any diagonal left-invariant metric as $g=(x_1,\dots,x_{\tfrac{n(n-1)}{2}})_{g_{\kil}}$.

\begin{proposition}
 Consider $G=\SO(n)$ and fix the Killing metric $g_{\kil}$ as a background metric. For $\left(g(t)=(x_1(t),\dots,x_{\tfrac{n(n-1)}{2}}(t))_{g_{\kil}},H_{\kil}\right)$  the generalized Ricci flow is given by
\begin{equation} \label{flow so}
x_k'(t)=- \sum_{i,j}(c_{ij}^k)^2+\frac{1}{2}\sum_{i,j}(c_{ij}^k)^2\frac{x_i^2+x_j^2-x_k^2}{x_ix_j}+\frac{1}{2}\sum_{i,j}\frac{(c_{ij}^k)^2}{x_ix_j} \quad \forall k= 1, \dots, \tfrac{n(n-1)}{2}.
 \end{equation}
\end{proposition}
\begin{proof}
 Replacing the formulas \eqref{Hk2} and \eqref{Rc} in \eqref{flow} the proposition follows as the basis $\beta$ is nice.
\end{proof}

It is clear that $g_{\kil}=(1,1,\dots,1)_{g_{\kil}}$ is a Bismut Ricci flat generalized metric, which means a fix point of the above system. To study the dynamical stability of this point, we consider the linearization of the given non-linear system.

\begin{theorem}
The Killing metic $g_{\kil}$ on the compact Lie group $\SO(n)$ is asymptotically stable for the dynamical system \eqref{flow so}.
\end{theorem}

\begin{proof}
 We define functions
 $$f_k(x_1,\dots,x_{\tfrac{n(n-1)}{2}}):=-\sum_{i,j}(c_{ij}^k)^2+\frac{1}{2}\sum_{i,j}(c_{ij}^k)^2\frac{x_i^2+x_j^2-x_k^2+1}{x_ix_j}, \quad \text{for } k= 1, \dots, \tfrac{n(n-1)}{2},$$
such that $x'(t)=f(x)$. Therefore,  the differential of $f$ is given by:
$$
\left\{\begin{aligned}
\frac{\partial f_k}{\partial x_k}=& \frac{1}{2}\sum_{i,j}\frac{(c_{ij}^k)^2}{x_ix_j}(-2x_k) \qquad \forall k= 1, \dots, \tfrac{n(n-1)}{2},\\
\frac{\partial f_k}{\partial x_i}=& \frac{1}{2}\sum_{i,j}(c_{ij}^k)^2\frac{x_i^2-x_j^2+x_k^2-1}{x_i^2x_j} \qquad \forall k\neq i. \\
\end{aligned} \right.
$$

Hence, on the Killing metric $g_{\kil}$,

$$Df(1,1,\dots,1)=\begin{pmatrix}
              -\sum_{i,j}(c_{ij}^k)^2 & 0 & \dots& \dots & 0 \\
              0 & -\sum_{i,j}(c_{ij}^k)^2& \dots& \dots& 0 \\
              0 & 0 & \ddots & \dots & 0 \\
              0 &  \dots& \dots& 0 & -\sum_{i,j}(c_{ij}^k)^2
\end{pmatrix}.$$

As all its eigenvalues are reals and negative the proof is complete.
\end{proof}

\end{document}